\documentclass[12pt]{article}

\usepackage{here}
\usepackage{pstricks}
\usepackage{pst-node}
\usepackage{amssymb}
\usepackage{amsthm}
\usepackage{amsmath}
\usepackage{graphicx}

\newtheorem{thm}{Theorem}
\newtheorem{lem}{Lemma}
\newtheorem{prop}{Proposition}
\newtheorem{cor}{Corollary}
\newtheorem{exm}{Example}

\begin{document}

\begin{center}
 {\Large \bf Domino tilings with diagonal impurities}

 \bigskip
 Fumihiko Nakano
\footnote{Faculty of Science, 
Department of Mathematics and Information Science,
Kochi University,
2-5-1, Akebonomachi, Kochi, 780-8520, Japan.
e-mail : 
nakano@math.kochi-u.ac.jp}
 and Taizo Sadahiro\footnote{
Faculty of Administration, 
Prefectural University of Kumamoto, 
Tsukide 3-1-100, Kumamoto, 862-8502, Japan.
e-mail : sadahiro@pu-kumamoto.ac.jp}

\end{center}

\begin{abstract}
This paper studies the dimer model 
on the dual graph of the square-octagon lattice, 
which can be viewed as the domino tilings with
impurities in some sense.
In particular, under a certain boundary condition,
we give an exact formula representing the probability of
finding an impurity at a given site in a uniformly random dimer
 configuration in terms of simple random walks on the square lattice.
\end{abstract}

\section{Introduction}
Although the dimer models on planar bipartite lattice graphs
have been greatly advanced 
over the last decade (see e.g., \cite{KOS},\cite{CKP}),
much less is known about non-bipartite cases.
This paper  deals with a non-bipartite lattice $\Gamma$,
the dual of the square-octagon lattice.
As will be clear later, the dimer model on $\Gamma$ can be
viewed as the domino tiling model containing certain impurities.
Our main aim in this paper is to study the behavior of these
impurities. 
In particular, under a certain boundary condition,
we give an exact formula representing the probability of
finding an impurity at a given site in a uniformly random dimer
 configuration in terms of the simple random walks on the square lattice.


\newcommand{\unitocatagon}{
   \psline(0.7071,-0.2929)(0.7071,0.2929)
   (0.2929,0.7071)(0.2929,0.7071)(-0.2929,0.7071)
   (-0.7071,0.2929)(-0.7071,-0.2929)
   \psline(-0.2929,-0.7071)(0.2929,-0.7071)
}
\begin{figure}[H]
 \begin{center}
 \begin{pspicture}(5,3)(-4,-2.2)
   \psset{unit=1.5}
  \begin{psclip}{
   \psframe[linecolor=white](-2.8,-1.8)(2.8,2.2)
   }
  \multiput(0,-4)(0,1){10}{\psline[linewidth=0.01](-10,0)(10,0)}
  \multiput(-4,0)(1,0){10}{\psline[linewidth=0.01](0,-10)(0,10)}
  \multiput(-8,0)(2,0){7}{\psline[linewidth=0.01](-10,10)(10,-10)}
  \multiput(-8,0)(2,0){7}{\psline[linewidth=0.01](10,10)(-10,-10)}
   \psset{linestyle=dashed, linecolor=gray,linewidth=0.01}
   \multiput(0,-6)(0,2){6}{\multiput(-6,0)(2,0){12}{\unitocatagon}}
   \psset{linestyle=dashed, linecolor=gray,linewidth=0.01}
   \multiput(1,-5)(0,2){5}{\multiput(-4,0)(2,0){11}{\unitocatagon}}

   \psset{linecolor=black,linestyle=solid,linewidth=0.02}
   \multiput(0,-4)(2,2){4}{
   \multiput(-9,0)(2,0){6}{\pscircle[fillstyle=solid,fillcolor=black](0,0){.07}}
   \multiput(-8,-1)(2,0){6}{\pscircle[fillstyle=solid,fillcolor=black](0,0){.07}}
   }
   \multiput(-1,-4)(2,2){4}{
   \multiput(-9,0)(2,0){6}{\pscircle[fillstyle=solid,fillcolor=white](0,0){.07}}
   \multiput(-8,-1)(2,0){6}{\pscircle[fillstyle=solid,fillcolor=white](0,0){.07}}
   }
  \end{psclip}
 \end{pspicture}
  \caption{Square-Octagon graph (dashed) and its dual }
  \label{fig:dualsqoct}
 \end{center}
\end{figure}
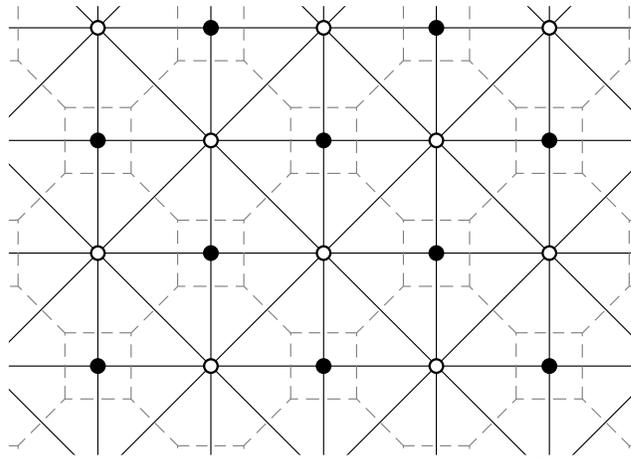

We define the {\em dual-square-octagon} graph $\Gamma$ as follows:
the vertices $V(\Gamma)$ is ${\mathbb Z}^2$ which is divided into  two subsets
$W = \{(x,y)\in V(\Gamma)~|~ x+y \mbox{ is even}\}$
and 
$B =  V(\Gamma)\backslash W$,
and there is an edge between  $v$ and $v'\in V(\Gamma)$ if and only if 
\[
 v-v' \in \left\{ \pm(1,0), \pm(0,1)\right\}
\]
or
\[
 v,v'\in W \mbox{ and } v-v'\in \{\pm(1,1), \pm(1,-1) \}.
\]
Thus, $\Gamma$ is the dual graph of the {\em square-octagon} lattice
graph (see Figure $\ref{fig:dualsqoct}$).
We say a vertex is white (resp. black) if it is in  $W$ (resp. $B$).
We call an edge connecting two white vertices a {\em diagonal} edge.
The edge set $E(\Gamma)$ of $\Gamma$ is divided into two disjoint subsets $E_1$ and $E_2$,
where $E_2$ is the set of diagonal edges and $E_1=E(\Gamma)\backslash E_1$.
Therefore, the graph $\Gamma$ is obtained from
the ordinary square lattice graph
by adding the edges $E_2$.
We denote by $\{v,v'\}$ the unoriented edge between two vertices $v$ and
$v'$.
In the following we sometimes need to orient the edges, and
we denote by $(v,v')$ the oriented edge {\em from $v$ to $v'$}.

A dimer covering (or perfect matching) $M$ of a graph $G=(V(G), E(G))$ is a
subset of the edge set $E(G)$ such that each element of the vertex set $V(G)$
is incident to exactly one element of $M$.
We call an edge $e$ in a dimer covering a {\em dimer}.
We say a subgraph $G$ of $\Gamma$ is {\em simply connected},
if $G$ and $\Gamma\backslash G$ are both connected.
We say a subgraph of $\Gamma$ is {\em normal}, if
it is simply connected and induced by a finite subset of $V(\Gamma)$.
This paper deals with  the dimer coverings of normal subgraphs of $\Gamma$.
A dimer covering  of a normal graph is equivalent to a tilings of the corresponding region
by square-octagon and octagon-octagon tiles (see Figure
$\ref{fig:matchingtiling}$).

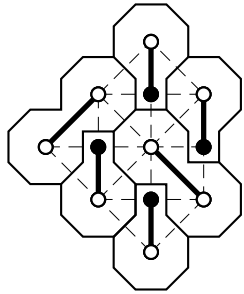
\begin{figure}[H]

 \begin{center}
  \psset{unit=0.7}
  \begin{pspicture}(4,4)(0,-1)

   \psline[linestyle=dashed,linewidth=0.01](0,2)(2,0)
   \psline[linestyle=dashed,linewidth=0.01](1,3)(3,1)
   \psline[linestyle=dashed,linewidth=0.01](0,2)(1,3)
   \psline[linestyle=dashed,linewidth=0.01](1,1)(2,2)
   \psline[linestyle=dashed,linewidth=0.01](2,0)(3,1)
   \psline[linestyle=dashed,linewidth=0.01](0,2)(2,2)
   \psline[linestyle=dashed,linewidth=0.01](1,1)(3,1)
   \psline[linestyle=dashed,linewidth=0.01](1,1)(1,3)
   \psline[linestyle=dashed,linewidth=0.01](2,0)(2,2)
   \psline[linestyle=dashed,linewidth=0.01](1,3)(2,4)(3,3)(3,1)
   \psline[linestyle=dashed,linewidth=0.01](3,3)(2,2)(2,4)
   \psline[linestyle=dashed,linewidth=0.01](3,3)(1,3)
   \psline[linestyle=dashed,linewidth=0.01](2,2)(3,2)

   \put(2,0){\pscircle[fillstyle=solid,fillcolor=white](0,0){0.15}}
   \put(1,1){\pscircle[fillstyle=solid,fillcolor=white](0,0){0.15}}
   \put(3,1){\pscircle[fillstyle=solid,fillcolor=white](0,0){0.15}}
   \put(0,2){\pscircle[fillstyle=solid,fillcolor=white](0,0){0.15}}
   \put(2,2){\pscircle[fillstyle=solid,fillcolor=white](0,0){0.15}}
   \put(1,3){\pscircle[fillstyle=solid,fillcolor=white](0,0){0.15}}
   \put(3,3){\pscircle[fillstyle=solid,fillcolor=white](0,0){0.15}}
   \put(2,4){\pscircle[fillstyle=solid,fillcolor=white](0,0){0.15}}

   \put(2,1){\pscircle[fillstyle=solid,fillcolor=black](0,0){0.15}}
   \put(1,2){\pscircle[fillstyle=solid,fillcolor=black](0,0){0.15}}
   \put(3,2){\pscircle[fillstyle=solid,fillcolor=black](0,0){0.15}}
   \put(2,3){\pscircle[fillstyle=solid,fillcolor=black](0,0){0.15}}

   \psline[linewidth=0.1]{*-*}(0,2)(1,3)
   \psline[linewidth=0.1]{*-*}(2,2)(3,1)
   \psline[linewidth=0.1]{*-*}(1,1)(1,2)
   \psline[linewidth=0.1]{*-*}(2,0)(2,1)
   \psline[linewidth=0.1]{*-*}(2,3)(2,4)
   \psline[linewidth=0.1]{*-*}(3,2)(3,3)

   \put(2,0){\pscircle[fillstyle=solid,fillcolor=white](0,0){0.15}}
   \put(1,1){\pscircle[fillstyle=solid,fillcolor=white](0,0){0.15}}
   \put(3,1){\pscircle[fillstyle=solid,fillcolor=white](0,0){0.15}}
   \put(0,2){\pscircle[fillstyle=solid,fillcolor=white](0,0){0.15}}
   \put(2,2){\pscircle[fillstyle=solid,fillcolor=white](0,0){0.15}}
   \put(1,3){\pscircle[fillstyle=solid,fillcolor=white](0,0){0.15}}
   \put(3,3){\pscircle[fillstyle=solid,fillcolor=white](0,0){0.15}}
   \put(2,4){\pscircle[fillstyle=solid,fillcolor=white](0,0){0.15}}

   \put(2,1){\pscircle[fillstyle=solid,fillcolor=black](0,0){0.15}}
   \put(1,2){\pscircle[fillstyle=solid,fillcolor=black](0,0){0.15}}
   \put(3,2){\pscircle[fillstyle=solid,fillcolor=black](0,0){0.15}}
   \put(2,3){\pscircle[fillstyle=solid,fillcolor=black](0,0){0.15}}

   \pspolygon(0.7071,1.7071)(0.7071,2.2929)(1.2929,2.2929)(1.2929,1.7071)
   (1.7071,1.2929)(1.7071,0.7071)(1.2929,0.2929)(0.7071,0.2929)(0.2929,0.7071)
   (0.2929,1.2929)
   \put(1,-1){\pspolygon(0.7071,1.7071)(0.7071,2.2929)(1.2929,2.2929)(1.2929,1.7071)
   (1.7071,1.2929)(1.7071,0.7071)(1.2929,0.2929)(0.7071,0.2929)(0.2929,0.7071)
   (0.2929,1.2929)}
   \psline(1.2929,2.2929)(1.7071,2.7071)(2.2929,2.7071)(2.7071,2.2929)(2.7071,1.7071)
   (3.2929,1.7071)(3.7071,1.2929)(3.7071,0.7071)(3.2929,0.2929)(2.7071,0.2929)
   \psline(0.2929,1.2929)(-0.2929,1.2929)(-0.7071,1.7071)(-0.7071,2.2929)(-0.2929,2.7071)
   (0.2929,2.7071)(0.2929,3.2929)(0.7071,3.7071)(1.2929,3.7071)(1.7071,3.2929)(1.7071,2.7071)
   \psline(1.2929,3.7071)(1.2929,4.2929)(1.7071,4.7071)(2.2929,4.7071)(2.7071,4.2929)(2.7071,3.7071)
   (2.2929,3.2929)(2.2929,2.7071)
   \psline(2.7071,3.7071)(3.2929,3.7071)(3.7071,3.2929)(3.7071,2.7071)(3.2929,2.2929)(3.2929,1.7071)

  \end{pspicture}
  \caption{Tiling and dimer covering}
  \label{fig:matchingtiling}
  \end{center}
 
\end{figure}

For a normal subgraph $G$ of $\Gamma$, we denote
\[
 W_G=V(G)\cap W, ~~ B_G=V(G)\cap B.
\]
Let $M$ be a dimer covering of a normal graph $G$ and
let $k$ be the number of diagonal edges in $M$. 
Then 
\begin{equation}\label{eq:numhyp}
 k =
  \frac{
 |W_G| - |B_G|
 }{2}.
\end{equation}
Hence, the number of diagonal edges in a dimer covering $M$ of $G$
is an invariant of $G$, not depending on the choice of $M$.
If a dimer covering $M$ of $G$ does not contain diagonal edge,
then it is a dimer covering of the ordinary square lattice graph,
also known as the {\em domino tiling}, which has been extensively
studied.
In this respect, it may be natural to call a dimer $e\in E_2$ of $G$
{\em impurity}.
Our main aim in this paper is to study the behavior of these impurities.
In our forthcoming paper it will be  shown  that the local transformations which
will be introduced as the $t$-moves and the $s$-moves in the next
section connects all dimer coverings, that is,
any dimer covering of a normal graph $G$ can be transformed into
any other dimer covering of $G$ by applying some sequence of the local transformations.
This property enables one to construct
an ergodic  Markov chain whose state space is the dimer coverings.
Figure \ref{fig:mcmc} shows the result of a simulation of the Markov chain
whose stationary distribution is uniform, where we can see that
the impurities tend to be located near
the diagonal edges on the boundary of the graph.
\begin{figure}[H]
 \begin{center}
   {\includegraphics[width=8cm]{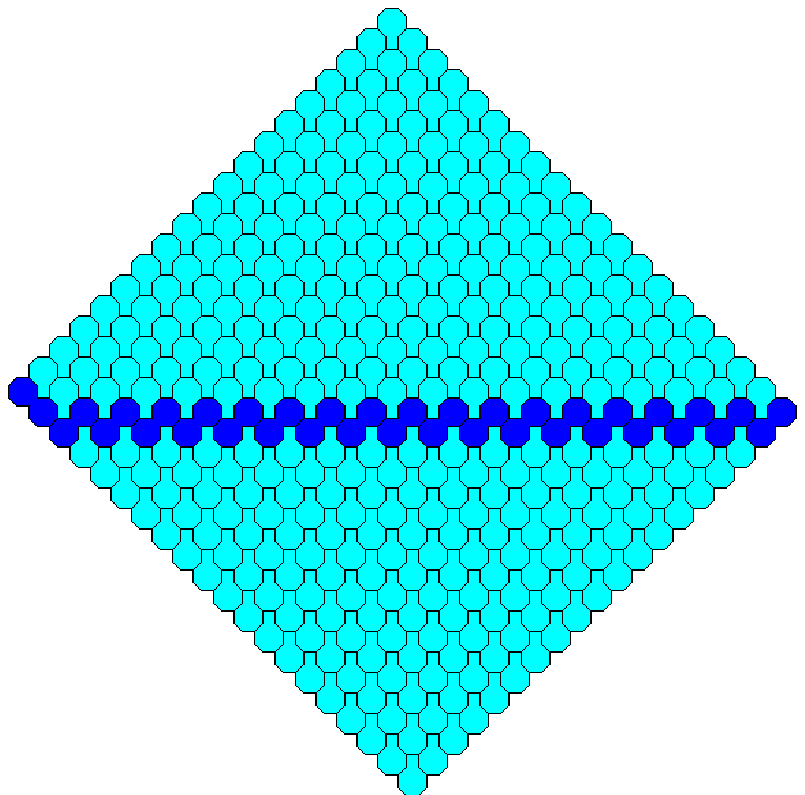}}
   {\includegraphics[width=8cm]{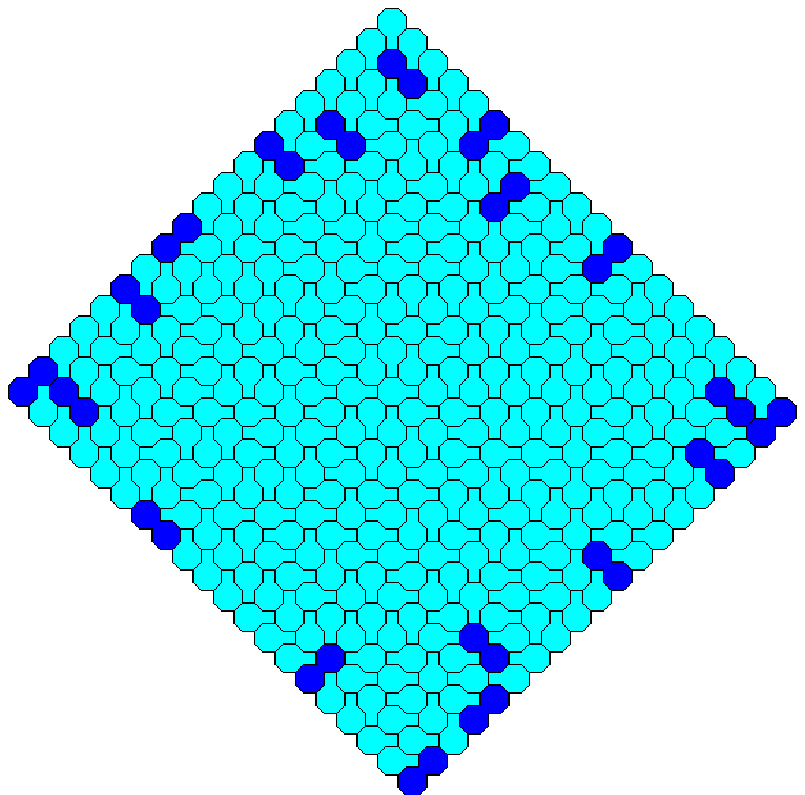}}
  \caption{Markov chain simulation:
  the initial configuration (left) and the configuration
  after $10^6$ steps (right)}
  \label{fig:mcmc}
 \end{center}
\end{figure}

The rest of this paper is organized as follows.
Section $\ref{sec:localmove}$ provides
some basic properties of the dimer model on $\Gamma$.
In Section $\ref{sec:oneimpurity}$,
we show an exact result on the easiest case where dimer coverings contains
exactly one impurity.

\section{Local moves and impurities' orbits}
\label{sec:localmove}
Let $\{a,b\}$ and $\{c,d\}$ be dimers contained in a dimer covering
$M$ of a normal graph $G$, which satisfy one of the followings:
\begin{description}
 \item[S:] $a,b,c,d$ are the four vertices of a unit square.
 \item[T:] $\{a,b\},\{b,c\}\in E_2$ and $\{c,d\}, \{d,a\}\in E_1$.
\end{description}
In case of $\bf S$ (resp. $\bf T$), 
we call the transformation which transforms $M$ into another 
by replacing $\{ \{a,b\}, \{c,d\} \}$ with
$\{\{b,c\},\{d,a\}\}$  an {\em $s$-move}
(resp. {\em $t$-move}),
which is shown in Figure $\ref{fig:localmoves}$.

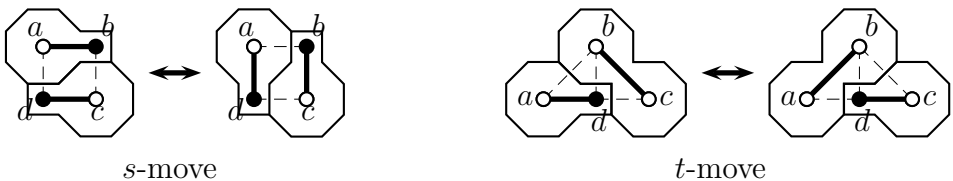
\begin{figure}[H]
 \begin{center}
     \begin{pspicture}(12,5)(-2,1.5)
   \psset{unit=0.7}
   \put(-1,4){
   \psline[linecolor=black,linewidth=0.01,linestyle=dashed](0,0)(0,1)
   \psline[linecolor=black,linewidth=0.01,linestyle=dashed](1,0)(1,1)
   \psline[linewidth=0.1]{*-*}(0,0)(1,0)
   \psline[linewidth=0.1]{*-*}(0,1)(1,1)
   \put(1,0){\pscircle[fillstyle=solid,fillcolor=white](0,0){0.15}}
   \put(0,1){\pscircle[fillstyle=solid,fillcolor=white](0,0){0.15}}
    \pspolygon(-0.2929,-0.2929)(-0.2929,0.2929)(0.2929,0.2929)(0.7071,0.7071)
    (1.2929,0.7071)(1.7071,0.2929)(1.7071,-0.2929)(1.2929,-0.7071)(0.7071,-0.7071)
    (0.2929,-0.2929)
    \put(0,1){
    \pspolygon(-0.2929,-0.7071)(-0.7071,-0.2929)(-0.7071,0.2929)
    (-0.2929,0.7071)(0.2929,0.7071)(0.7071,0.2929)(1.2929,0.2929)(1.2929,-0.2929)
    (0.7071,-0.2929)(0.2929,-0.7071)
    }
      \put(-.3,1.2){$a$}
      \put(1.1,1.2){$b$}
      \put(.9,-.4){$c$}
      \put(-.5,-.4){$d$}

   \put(4,0){
   \psline[linecolor=black,linewidth=0.01,linestyle=dashed](0,0)(1,0)
   \psline[linecolor=black,linewidth=0.01,linestyle=dashed](0,1)(1,1)
   \psline[linewidth=0.1]{*-*}(0,0)(0,1)
   \psline[linewidth=0.1]{*-*}(1,0)(1,1)
   \put(1,0){\pscircle[fillstyle=solid,fillcolor=white](0,0){0.15}}
   \put(0,1){\pscircle[fillstyle=solid,fillcolor=white](0,0){0.15}}
    \put(0,-1){
    \pspolygon(0.7071,1.7071)(0.7071,2.2929)(1.2929,2.2929)(1.2929,1.7071)
   (1.7071,1.2929)(1.7071,0.7071)(1.2929,0.2929)(0.7071,0.2929)(0.2929,0.7071)
   (0.2929,1.2929)
    }
    \put(0,0){
    \psline(0.2929,-0.2929)(-0.2929,-0.2929)(-0.2929,0.2929)(-0.7071,0.7071)(-0.7071,1.2929)
    (-0.2929,1.7071)(0.2929,1.7071)(0.7071,1.2929)
    }
      \put(-.3,1.2){$a$}
      \put(1.1,1.2){$b$}
      \put(.9,-.4){$c$}
      \put(-.5,-.4){$d$}
    }
   \psline[linewidth=0.1]{<->}(2,0.5)(3,0.5)
   \put(1.5,-1.5){$s$-move}
    }

   \put(9,4){
    \put(-0.5,0){
   \psline[linecolor=black,linewidth=0.01,linestyle=dashed](1,0)(2,0)
   \psline[linecolor=black,linewidth=0.01,linestyle=dashed](0,0)(1,1)
   \psline[linecolor=black,linewidth=0.01,linestyle=dashed](1,0)(1,1)
   \psline[linewidth=0.1]{*-*}(0,0)(1,0)
   \psline[linewidth=0.1]{*-*}(2,0)(1,1)
   \put(0,0){\pscircle[fillstyle=solid,fillcolor=white](0,0){0.15}}
   \put(1,1){\pscircle[fillstyle=solid,fillcolor=white](0,0){0.15}}
   \put(2,0){\pscircle[fillstyle=solid,fillcolor=white](0,0){0.15}}
    \put(0,0){
    \pspolygon(-0.2929,-0.7071)(-0.7071,-0.2929)(-0.7071,0.2929)
    (-0.2929,0.7071)(0.2929,0.7071)(0.7071,0.2929)(1.2929,0.2929)(1.2929,-0.2929)
    (0.7071,-0.2929)(0.2929,-0.7071)
    }
    \put(1,1){
    \pspolygon(-0.7071,-0.2929)(-0.7071,0.2929)(-0.2929,0.7071)(0.2929,0.7071)
    (0.7071,0.2929)(0.7071,-0.2929)(1.2929,-0.2929)(1.7071,-0.7071)(1.7071,-1.2929)
    (1.2929,-1.7071)(0.7071,-1.7071)(0.2929,-1.2929)(0.2929,-0.7071)(-0.2929,-0.7071)
    }
      \put(-.5,-.1){$a$}
      \put(1.1,1.2){$b$}
      \put(2.2,-.1){$c$}
      \put(.9,-.6){$d$}
    }

   \put(4.5,0){
   \psline[linecolor=black,linewidth=0.01,linestyle=dashed](0,0)(1,0)
   \psline[linecolor=black,linewidth=0.01,linestyle=dashed](2,0)(1,1)
   \psline[linecolor=black,linewidth=0.01,linestyle=dashed](1,0)(1,1)
   \psline[linewidth=0.1]{*-*}(1,0)(2,0)
   \psline[linewidth=0.1]{*-*}(0,0)(1,1)
   \put(0,0){\pscircle[fillstyle=solid,fillcolor=white](0,0){0.15}}
   \put(1,1){\pscircle[fillstyle=solid,fillcolor=white](0,0){0.15}}
   \put(2,0){\pscircle[fillstyle=solid,fillcolor=white](0,0){0.15}}
    \put(1,0){
        \pspolygon(-0.2929,-0.2929)(-0.2929,0.2929)(0.2929,0.2929)(0.7071,0.7071)
    (1.2929,0.7071)(1.7071,0.2929)(1.7071,-0.2929)(1.2929,-0.7071)(0.7071,-0.7071)
    (0.2929,-0.2929)
    }
    \put(0,0){
    \pspolygon(-0.2929,-0.7071)(-0.7071,-0.2929)(-0.7071,0.2929)
    (-0.2929,0.7071)(0.2929,0.7071)(0.2929,1.2929)(0.7071,1.7071)
    (1.2929,1.7071)(1.7071,1.2929)(1.7071,0.7071)(1.2929,0.2929)
    (0.7071,0.2929)(0.7071,-0.2929)(0.2929,-0.7071)
    }
      \put(-.5,-.1){$a$}
      \put(1.1,1.2){$b$}
      \put(2.2,-.1){$c$}
      \put(.9,-.6){$d$}

   }

   \psline[linewidth=0.1]{<->}(2.5,0.5)(3.5,0.5)
   \put(2,-1.5){$t$-move}
   }

  \end{pspicture}
  \caption{local moves}
  \label{fig:localmoves}
 \end{center}
\end{figure}

We divide the white vertices $W$ into two parts
$W_0=2{\mathbb Z}\times 2{\mathbb Z}$ and
$W_1=W_0+(1,1)$ and define
two graphs $\Lambda$ and $\Lambda^\perp$ as follows:
$\Lambda$ has vertices $W_0$ and it has an edge
between $v$ and $v'\in W_0$ if and only if $v-v'\in\{\pm(2,0), \pm(0,2)\}$,
$\Lambda^\perp$ is the dual graph of $\Lambda$ having
vertices $W_1$.
Let $w_1,w_2, w_3, w_4$ be four white vertices
which are adjacent to a black vertex $b$ 
listed in counter-clockwise order
as shown in Figure $\ref{fig:arc-dimer}$.
Then one of the two sets $\{w_1,w_3\}$
and $\{w_2, w_4\}$ is contained in $W_0$ and
the other is in $W_1$.
Let us assume that
$w_1, w_2$ and $b$ are contained in a normal graph $G$
and $w_3$ and $w_4$ are not necessarily contained in $G$.
For a dimer covering $M$ of $G$ we draw
an arc centered at $w_1$ (resp. $w_2$) 
which starts at the middle point of the edge $\{w_1, w_2\}$
and ends at a point on the edge $\{w_1,b\}$ (resp. $\{w_2,b\}$)
if $\{w_2,b\}$ or $\{w_4,b\}$ (resp. $\{w_1,b\}$ or $\{w_3,b\}$)
is contained in $M$.
Then a dimer covering $M$ of $G$ defines
curves on the plane composed of these arcs, which we call the {\em
slit-curves}.
Figure $\ref{fig:slit-curves}$ shows an example of slit-curves.

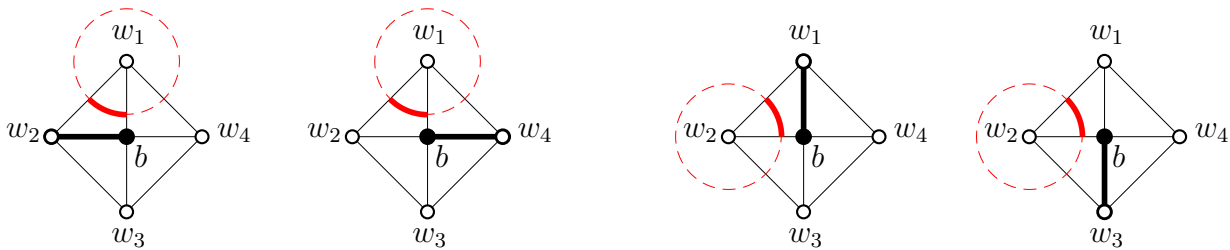
\begin{figure}[H]
 \begin{center}
  \begin{pspicture}(14,3)
   \put(0,0){
   \psline[linewidth=0.01](1,0)(2,1)(1,2)(0,1)(1,0)(1,2)
   \psline[linewidth=0.01](0,1)(2,1)
   \psline[linewidth=0.075,linecolor=black]{o-o}(1,1)(0,1)
   \put(1,2){\put(-0.2,0.3){$w_1$}}
   \put(0,1){\put(-0.6,0){$w_2$}}
   \put(1,0){\put(-0.2,-0.4){$w_3$}}
   \put(2,1){\put(0.2,0){$w_4$}}
   \put(1,1){\put(0.1,-0.4){$b$}}
   \put(1,0){\pscircle[fillstyle=solid,fillcolor=white](0,0){0.1}}
   \put(2,1){\pscircle[fillstyle=solid,fillcolor=white](0,0){0.1}}
   \put(1,2){\pscircle[fillstyle=solid,fillcolor=white](0,0){0.1}}
   \put(0,1){\pscircle[fillstyle=solid,fillcolor=white](0,0){0.1}}
   \put(1,1){\pscircle[fillstyle=solid,fillcolor=black](0,0){0.1}}
   \psarc[linecolor=red,linewidth=0.075](1,2){0.707}{225}{270}
   \pscircle[linecolor=red,linewidth=0.01,linestyle=dashed](1,2){0.707}
   }
   \put(4,0){
   \psline[linewidth=0.01](1,0)(2,1)(1,2)(0,1)(1,0)(1,2)
   \psline[linewidth=0.01](0,1)(2,1)
   \psline[linewidth=0.075,linecolor=black]{o-o}(1,1)(2,1)
   \put(1,2){\put(-0.2,0.3){$w_1$}}
   \put(0,1){\put(-0.6,0){$w_2$}}
   \put(1,0){\put(-0.2,-0.4){$w_3$}}
   \put(2,1){\put(0.2,0){$w_4$}}
   \put(1,1){\put(0.1,-0.4){$b$}}
   \put(1,0){\pscircle[fillstyle=solid,fillcolor=white](0,0){0.1}}
   \put(2,1){\pscircle[fillstyle=solid,fillcolor=white](0,0){0.1}}
   \put(1,2){\pscircle[fillstyle=solid,fillcolor=white](0,0){0.1}}
   \put(0,1){\pscircle[fillstyle=solid,fillcolor=white](0,0){0.1}}
   \put(1,1){\pscircle[fillstyle=solid,fillcolor=black](0,0){0.1}}
   \psarc[linecolor=red,linewidth=0.075](1,2){0.707}{225}{270}
   \pscircle[linecolor=red,linewidth=0.01,linestyle=dashed](1,2){0.707}
   }

   \put(9,0){
   \psline[linewidth=0.01](1,0)(2,1)(1,2)(0,1)(1,0)(1,2)
   \psline[linewidth=0.01](0,1)(2,1)
   \psline[linewidth=0.075,linecolor=black]{o-o}(1,1)(1,2)
   \put(1,2){\put(-0.2,0.3){$w_1$}}
   \put(0,1){\put(-0.6,0){$w_2$}}
   \put(1,0){\put(-0.2,-0.4){$w_3$}}
   \put(2,1){\put(0.2,0){$w_4$}}
   \put(1,1){\put(0.1,-0.4){$b$}}
   \put(1,0){\pscircle[fillstyle=solid,fillcolor=white](0,0){0.1}}
   \put(2,1){\pscircle[fillstyle=solid,fillcolor=white](0,0){0.1}}
   \put(1,2){\pscircle[fillstyle=solid,fillcolor=white](0,0){0.1}}
   \put(0,1){\pscircle[fillstyle=solid,fillcolor=white](0,0){0.1}}
   \put(1,1){\pscircle[fillstyle=solid,fillcolor=black](0,0){0.1}}
   \psarc[linecolor=red,linewidth=0.075](0,1){0.707}{0}{45}
   \pscircle[linecolor=red,linewidth=0.01,linestyle=dashed](0,1){0.707}
   }

   \put(13,0){
   \psline[linewidth=0.01](1,0)(2,1)(1,2)(0,1)(1,0)(1,2)
   \psline[linewidth=0.01](0,1)(2,1)
   \psline[linewidth=0.075,linecolor=black]{o-o}(1,1)(1,0)
   \put(1,2){\put(-0.2,0.3){$w_1$}}
   \put(0,1){\put(-0.6,0){$w_2$}}
   \put(1,0){\put(-0.2,-0.4){$w_3$}}
   \put(2,1){\put(0.2,0){$w_4$}}
   \put(1,1){\put(0.1,-0.4){$b$}}
   \put(1,0){\pscircle[fillstyle=solid,fillcolor=white](0,0){0.1}}
   \put(2,1){\pscircle[fillstyle=solid,fillcolor=white](0,0){0.1}}
   \put(1,2){\pscircle[fillstyle=solid,fillcolor=white](0,0){0.1}}
   \put(0,1){\pscircle[fillstyle=solid,fillcolor=white](0,0){0.1}}
   \put(1,1){\pscircle[fillstyle=solid,fillcolor=black](0,0){0.1}}
   \psarc[linecolor=red,linewidth=0.075](0,1){0.707}{0}{45}
   \pscircle[linecolor=red,linewidth=0.01,linestyle=dashed](0,1){0.707}
   }

  \end{pspicture}
 \end{center}
 \caption{Arcs and dimers}
 \label{fig:arc-dimer}
\end{figure}

\begin{figure}[H]
\newcommand{\vblock}{
   \put(-1,0){\psarc[linecolor=red](0,0){0.707}{-45}{45}}
   \put(1,0){\psarc[linecolor=red](0,0){0.707}{135}{225}}
}
\newcommand{\hblock}{
   \put(0,-1){\psarc[linecolor=red](0,0){0.707}{45}{135}}
   \put(0,1){\psarc[linecolor=red](0,0){0.707}{225}{315}}
}
 \begin{center}
  \begin{pspicture}(15,3)(0,-4)

      \begin{psclip}{\pspolygon[linewidth=0](0,0)(3,3)(7,-1)(4,-4)}
       \multiput(1,-1)(1,-1){3}{\psline(0,0)(3,3)}
       \multiput(1,1)(1,1){2}{\psline(0,0)(4,-4)}
    \multiput(0,-6)(0,2){8}{
    \psline(0,0)(8,0)
    }
    \multiput(0,-6)(2,0){8}{
    \psline(0,0)(0,8)
    }

    \multiput(1,-5)(0,2){8}{
    \psline[linestyle=solid](0,0)(8,0)
    }
    \multiput(1,-5)(2,0){8}{
    \psline[linestyle=solid](0,0)(0,8)
    }
   \end{psclip}

   \psline[linecolor=blue,linewidth=0.25]{*-*}(0,0)(1,1)
   \psline[linecolor=blue,linewidth=0.25]{*-*}(3,3)(4,2)
   \psline[linecolor=blue,linewidth=0.25]{*-*}(5,-1)(4,-2)
   \psline[linecolor=blue,linewidth=0.25]{*-*}(4,-4)(5,-3)

   \psset{dotsize=.35}

   \psline[linecolor=cyan,linewidth=0.25]{*-*}(1,0)(2,0)
   \psline[linecolor=cyan,linewidth=0.25]{*-*}(4,0)(4,-1)
   \psline[linecolor=cyan,linewidth=0.25]{*-*}(5,0)(6,0)
   \psline[linecolor=cyan,linewidth=0.25]{*-*}(2,1)(2,2)
   \psline[linecolor=cyan,linewidth=0.25]{*-*}(2,1)(2,2)
   \psline[linecolor=cyan,linewidth=0.25]{*-*}(3,1)(3,2)
   \psline[linecolor=cyan,linewidth=0.25]{*-*}(4,1)(5,1)
   \psline[linecolor=cyan,linewidth=0.25]{*-*}(2,-1)(1,-1)
   \psline[linecolor=cyan,linewidth=0.25]{*-*}(3,0)(3,-1)
   \psline[linecolor=cyan,linewidth=0.25]{*-*}(6,-1)(7,-1)
   \psline[linecolor=cyan,linewidth=0.25]{*-*}(6,-2)(5,-2)
   \psline[linecolor=cyan,linewidth=0.25]{*-*}(3,-3)(4,-3)
   \psline[linecolor=cyan,linewidth=0.25]{*-*}(2,-2)(3,-2)

   \put(3,2){\vblock}
   \put(2,1){\vblock}
   \put(4,1){\hblock}
   \put(1,0){\hblock}
   \put(3,0){\vblock}
   \put(5,0){\hblock}
   \put(2,-1)\hblock
   \put(4,-1)\vblock
   \put(6,-1)\hblock
   \multiput(3,-2)(2,0){2}{
   \hblock
   }
   \put(4,-3)\hblock

   \put(8,0)
   {
      \begin{psclip}{\pspolygon[linewidth=0](0,0)(3,3)(7,-1)(4,-4)}
       \psset{linewidth=0.01}
       \multiput(1,-1)(1,-1){3}{\psline(0,0)(3,3)}
       \multiput(1,1)(1,1){2}{\psline(0,0)(4,-4)}
    \multiput(0,-6)(0,2){8}{
    \psline(0,0)(8,0)
    }
    \multiput(0,-6)(2,0){8}{
    \psline(0,0)(0,8)
    }
    \multiput(1,-5)(0,2){8}{
    \psline[linestyle=solid](0,0)(8,0)
    }
    \multiput(1,-5)(2,0){8}{
    \psline[linestyle=solid](0,0)(0,8)
    }
   \end{psclip}


   \psset{dotsize=.35}


   \put(3,2){\vblock}
   \put(2,1){\vblock}
   \put(4,1){\hblock}
   \put(1,0){\hblock}
   \put(3,0){\vblock}
   \put(5,0){\hblock}
   \put(2,-1)\hblock
   \put(4,-1)\vblock
   \put(6,-1)\hblock
   \multiput(3,-2)(2,0){2}{
   \hblock
   }
   \put(4,-3)\hblock

   \psline[linewidth=0.1](0,0)(2,0)(2,2)
   \psline[linewidth=0.1,linestyle=dashed](1,-1)(3,-1)(3,3)
   \psline[linewidth=0.1,linestyle=dashed](3,1)(5,1)
   \psline[linewidth=0.1](2,-2)(6,-2)
   \psline[linewidth=0.1](4,-2)(4,0)(6,0)
   \psline[linewidth=0.1,linestyle=dashed](5,-1)(7,-1)
   \psline[linewidth=0.1,linestyle=dashed](3,-3)(5,-3)
   \multiput(0,0)(1,1){4}{
   \multiput(0,0)(1,-1){5}{
   \pscircle[fillstyle=solid,fillcolor=white](0,0){.1}
   }
   }
   }

  \end{pspicture}
 \end{center}
 \caption{Left: Slit-curves generated by a dimer covering.
 Right: Primary forest (solid) and dual forest (dashed) }
 \label{fig:slit-curves}
\end{figure}
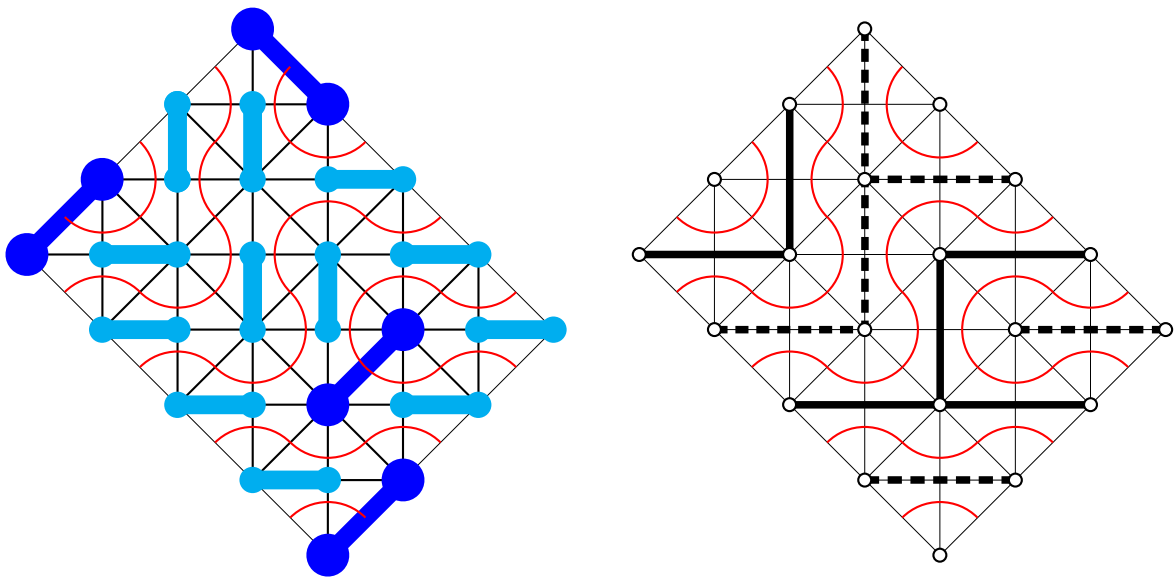

When a $t$-move $\tau$ transforms a dimer covering 
$M$ into another dimer covering $M'$ by
replacing an impurity $e$ with another impurity $e'$,
we simply say $\tau$ transforms $e$ to $e'$ and simply denote as
\[
 e'=\tau(e).
\]
\begin{prop}
 \label{prop:tkeepsslit}
A $t$-move  keeps all slit-curves unchanged.
There exists a sequence of $t$-moves which transforms
an impurity $e$ to another impurity $e'\in E_2$ if and only if
$e$ and $e'$ intersects with a common slit-curve.
\end{prop}
\begin{proof}
 From the definition, a $t$-move clearly keeps slit-curves
 unchanged.
 It is clear that an impurity  $e$ can be transformed to $e'$
 by one $t$-move if and only if $e$ and $e'$ have a common terminal vertex
 and there exists a slit-curve
 which intersects with  both of $e$ and $e'$.
 Now the last statement can be easily proved by the induction on the
 length of the portion of the slit-curve from $e$ to $e'$.
\end{proof}

\begin{cor}
 \label{cor:slit-term}
 If a slit-curve does not terminate on a diagonal edge,
 it does not intersects with impurities.
 For each slit-curve $C$, there is at most one impurity
 intersecting with $C$.
\end{cor}
\begin{proof}
 Assume that a slit-curve terminates on an edge $\{v_1,v_2\}\in E_1$ as shown in
 Figure $\ref{fig:nonimpslit}$.
 Then the dimer $\{v_2,v_3\}$ must be contained in the dimer covering
 which generate the slit-curves.
 For the sake of contradiction, assume that
 an impurity intersects with this slit-curve.
 The impurity can not be transformed to $\{v_1,v_3\}$ in the figure
 since the vertex $v_3$ must be incident to a dimer $\{v_2,v_3\}$,
 which contradicts Proposition $\ref{prop:tkeepsslit}$.
 The last statement follows from
 Proposition $\ref{prop:tkeepsslit}$ and the fact that
 a $t$-move can  transform only one impurity.
  
 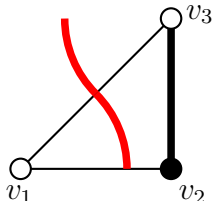
\begin{figure}[H]
  \begin{center}
   \begin{pspicture}(2,2)(0,0)
    \psline[]{*-*}(0,0)(2,0)
    \psline[]{*-*}(2,0)(2,2)
    \psline[]{*-*}(0,0)(2,2)
    \psline[linewidth=.1]{*-*}(2,0)(2,2)
    \pscircle[fillstyle=solid,fillcolor=white](0,0){.15}
    \pscircle[fillstyle=solid,fillcolor=white](2,2){.15}
    \pscircle[fillstyle=solid,fillcolor=black](2,0){.15}
    \psarc[linewidth=.1,linecolor=red](0,0){1.414}{0}{45}
    \psarc[linewidth=.1,linecolor=red](2,2){1.414}{180}{225}
    \put(-0.2,-0.4){$v_1$}
    \put(2.1,-.4){$v_2$}
    \put(2.2,2){$v_3$}
   \end{pspicture}
   \caption{A slit-curve terminating on an edge in $E_1$}
   \label{fig:nonimpslit}
  \end{center}
 \end{figure}
\end{proof}

\begin{prop}
 \label{prop:slit-curve}
 \begin{enumerate}
  \item \label{prop:slit-curve-1}A slit-curve does not intersects with another slit-curve.
  \item A slit-curve does not form a loop.
 \end{enumerate}
\end{prop}
\begin{proof}
 From the definition of the slit-curves, $\ref{prop:slit-curve-1}$ is
 obvious.
 For the sake of the contradiction, assume that a slit-curve $C$ form a
 loop. Without loss of generality we may assume
 $C$ is the innermost loop.
 Then the subgraph of $G$ induced by the vertices contained inside $C$
 is a tree, since otherwise a slit-curve exists inside $C$ and
 it must form a loop.
 Since the tree $T$ inside $C$ can not have black leaves, i.e.,
 black vertices which is incident to only one edge of $T$  and
 hence $T$ has odd number of
 vertices, $n$ whites and $n-1$ blacks.
 Thus $C$ must intersect with an impurity.
 Let us denote by $B_C$ (resp. $W_C$) the black (resp. white)
 vertices which are outside of $C$ and adjacent to vertices inside $C$.
 Then by the induction on the number of white vertices
 inside $C$, we have $|B_C|=|W_C|$. (See Figure $\ref{fig:slit-loop}$.)
 Let $b\in B_C$. Then there exists exactly one vertex $w\in W_C$ such
 that $\{b,w\}\in M$.
 Therefore every element of $W_C$ is incident to a dimer outside
 of $C$, which is a contradiction.
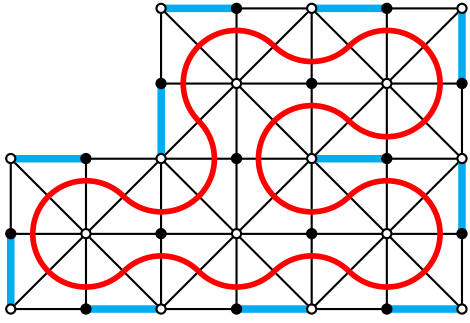
\begin{figure}[H]
 \begin{center}
  \begin{pspicture}(7,4)

   \psline(0,0)(6,0)
   \psline(0,1)(6,1)
   \psline(0,2)(6,2)
   \psline(2,3)(6,3)
   \psline(2,4)(6,4)
   \psline(0,0)(0,2)
   \psline(1,0)(1,2)
   \multiput(2,0)(1,0){5}{\psline(0,0)(0,4)}
   \multiput(0,0)(2,0){2}{\psline(0,0)(4,4)}
   \psline(4,0)(6,2)
   \psline(2,0)(0,2)\psline(4,0)(2,2)\psline(6,0)(2,4)\psline(6,2)(4,4)

   \multiput(0,0)(2,0){3}{\psline[linecolor=cyan,linewidth=0.1](1,0)(2,0)}
   \multiput(6,0)(0,2){2}{\psline[linecolor=cyan,linewidth=0.1](0,1)(0,2)}
   \multiput(6,4)(-2,0){2}{\psline[linecolor=cyan,linewidth=0.1](-1,0)(-2,0)}
   \psline[linecolor=cyan,linewidth=0.1](2,3)(2,2)
   \psline[linecolor=cyan,linewidth=0.1](1,2)(0,2)
   \psline[linecolor=cyan,linewidth=0.1](0,0)(0,1)
   \psline[linecolor=cyan,linewidth=0.1](4,2)(5,2)

   \psarc[linewidth=0.075,linecolor=red](1,1){0.707}{45}{315}
   \psarc[linewidth=0.075,linecolor=red](2,0){0.707}{45}{135}
   \psarc[linewidth=0.075,linecolor=red](3,1){0.707}{225}{315}
   \psarc[linewidth=0.075,linecolor=red](4,0){0.707}{45}{135}
   \psarc[linewidth=0.075,linecolor=red](5,1){0.707}{-135}{135}
   \psarc[linewidth=0.075,linecolor=red](4,2){0.707}{45}{315}
   \psarc[linewidth=0.075,linecolor=red](5,3){0.707}{-135}{135}
   \psarc[linewidth=0.075,linecolor=red](4,4){0.707}{225}{315}
   \psarc[linewidth=0.075,linecolor=red](3,3){0.707}{45}{225}
   \psarc[linewidth=0.075,linecolor=red](2,2){0.707}{-135}{45}
   \multiput(0,0)(0,2){2}{
   \multiput(0,0)(2,0){4}{\pscircle[fillstyle=solid,fillcolor=white](0,0){.075}}
   }
   \multiput(2,4)(2,0){3}{\pscircle[fillstyle=solid,fillcolor=white](0,0){.075}}
   \multiput(1,1)(2,0){3}{\pscircle[fillstyle=solid,fillcolor=white](0,0){.075}}
   \multiput(3,3)(2,0){2}{\pscircle[fillstyle=solid,fillcolor=white](0,0){.075}}

   \multiput(0,0)(0,2){2}{
   \multiput(1,0)(2,0){3}{\pscircle[fillstyle=solid,fillcolor=black](0,0){.075}}
   }
   \multiput(3,4)(2,0){2}{\pscircle[fillstyle=solid,fillcolor=black](0,0){.075}}
   \multiput(0,1)(2,0){4}{\pscircle[fillstyle=solid,fillcolor=black](0,0){.075}}
   \multiput(2,3)(2,0){3}{\pscircle[fillstyle=solid,fillcolor=black](0,0){.075}}

  \end{pspicture}
  \caption{If a slit-curve $C$ formed  a loop, there would be an odd number
  of vertices 
  inside $C$ and no vacant vertex around $C$.}
  \label{fig:slit-loop}
 \end{center}
\end{figure}
\end{proof}
\begin{prop}
 \label{prop:curve-forrest}
Remove the edges of $G$ which intersect with slit-curves.
Then, each of the connected components of the resulting graph is a tree,
 that is, the slit-curves determine a spanning forest of $G$.
\end{prop}
\begin{proof}
 If a connected component of the graph obtained by removing
 edges intersecting with slit-curves contains a loop,
 then a slit-curve must form a loop, which contradicts
 Proposition $\ref{prop:slit-curve}$.
\end{proof}

Each tree in the forest obtained by removing edges intersecting with
 slit-curves does not contain a path of the form like 
\begin{pspicture}(0.6,0.5)(-0.1,0)
 \psline(0,0)(0.5,0)(0.5,0.5)
 \pscircle[fillstyle=solid,fillcolor=white](0,0){.075}
 \pscircle[fillstyle=solid,fillcolor=black](0.5,0){.075}
 \pscircle[fillstyle=solid,fillcolor=white](.5,0.5){.075}
\end{pspicture}
, i.e., a path bended at a black vertex, 
hence it can be viewed as a tree in $\Lambda$ or $\Lambda^\perp$,
by removing the edges incident to black leaves, i.e., the 
black vertices each of which is incident
to exactly one edge in the tree.
We call the set of these trees in $\Lambda$ 
(resp. $\Lambda^\perp$) the {\em primary  forest}
(resp. dual forest) 
obtained from $M$ and denote it by $F(M)$ (resp. $F^\perp(M)$).

 Let  ${\mathcal M}(G)$ denote
 the set of dimer coverings of a normal graph $G$.
 We introduce a relation $\stackrel{t}{\sim}$ on ${\mathcal M}(G)$
 as follows:
 For two dimer coverings $M_1, M_2\in{\mathcal M}(G)$,  we define
 \[
  M_1\stackrel{t}{\sim} M_2
 \]
 if $M_1$ can be transformed into $M_2$ by applying
 a sequence of $t$-moves.
 Then the relation $\stackrel{t}{\sim}$ is clearly a equivalence
 relation and the set of equivalence classes
 are denoted by ${\mathcal M}(G)/\stackrel{t}{\sim}$.

\section{Configurations with only one impurity}
\label{sec:oneimpurity}
In this section, we show an exact enumerative formula
of dimer coverings of graphs of special shapes, each 
of which has dimer coverings with only one impurity.

\subsection{Temperley bijection}
Let $H$ be a simply connected subgraph of $\Lambda$.
We define the dual graph $H^\perp$ of $H$ in the following way.
$H^\perp$ has the vertices $V(H^\perp)$  consisting of vertices
corresponding to the faces of $H$, more specifically,
\[
 V(H^\perp)= \{(x,y)\in W_1=(1,1)+2{\mathbb Z}\times2{\mathbb Z}~|~
 (x,y) \mbox{ is in a face of }H\}\cup\{f^*\},
\]
where $f^*$ is the vertex 
taken from $W_1$ so that the graph induced by $V(H^\perp)$
in $\Lambda^\perp$ is simply connected in $\Lambda$.
Let $d^*$ be the number of edges $l_i$ of $\Lambda$ which connect
$f^*$ and other vetices in $V(H^\perp)$.
Then, since $H$ is simply connected, $d^*$ is at most $3$.
$H^\perp$ has edges $E(H^\perp)$ each of which corresponds to an edge
bounding a face of $H$.
Then, by embedding $H$ and $H^\perp$ simultaneously into the plane
so that each edge $e^\perp$ of $H^\perp$ crosses the corresponding edge
$e$ of $H$ only once at the middle point of $e$,
where we add a new black vertex,
and we obtain a bipartite graph $N'$.

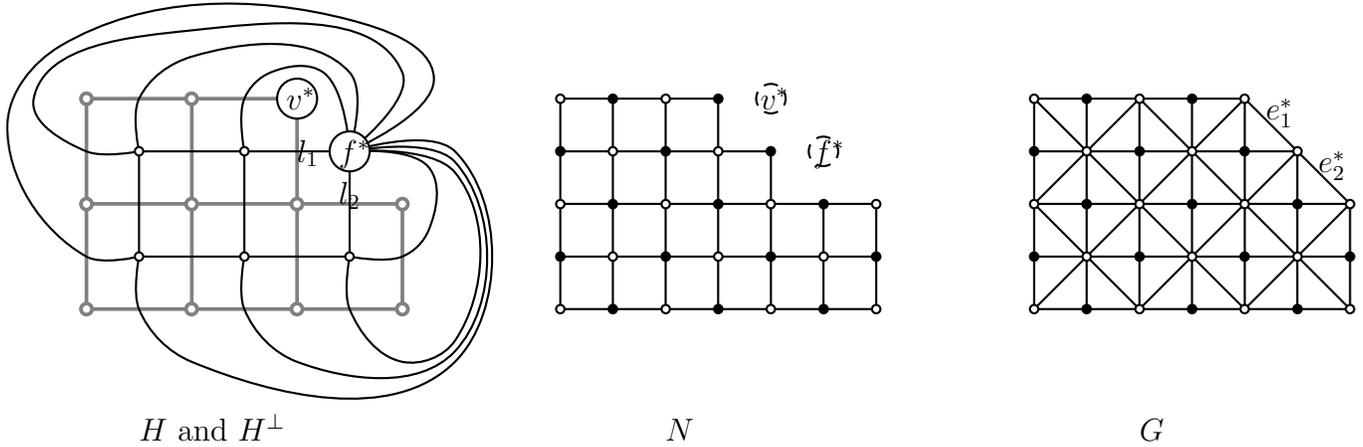
\begin{figure}[H]
 \begin{center}
  \begin{pspicture}(12,4)(0,-2)
   \psset{unit=.7}
   \put(-2,0){

   {\psset{linewidth=.07,linecolor=gray}
   \psline(0,0)(6,0)
   \psline(0,2)(6,2)
   \psline(0,4)(4,4)
   \psline(0,0)(0,4)
   \psline(2,0)(2,4)
   \psline(4,0)(4,4)
   \psline(6,0)(6,2)
   \multiput(0,0)(2,0){4}{
   \pscircle[fillstyle=solid,fillcolor=white](0,0){.14}
   }
   \multiput(0,2)(2,0){4}{
   \pscircle[fillstyle=solid,fillcolor=white](0,0){.14}
   }
   \multiput(0,4)(2,0){3}{
   \pscircle[fillstyle=solid,fillcolor=white](0,0){.14}
   }
   }
   \pscircle[fillstyle=solid,fillcolor=white](4,4){.4}
   \put(3.8,3.8){$v^*$}

   \put(1,1){
   \psline(0,0)(4,0)
   \psline(0,2)(4,2)
   \psline(0,0)(0,2)
   \psline(2,0)(2,2)
   \psline(4,0)(4,2)
   \pscurve(0,0)(0,-1)(1,-2)(6,-2)(6,2)(4,2)
   \pscurve(2,0)(2,-1)(3,-2)(6,-1.8)(6,1.8)(4,2)
   \pscurve(4,0)(4,-1)(5,-2)(6,-1.6)(6,1.6)(4,2)
   \pscurve(4,0)(5,0)(5.5,.5)(5.5,1.7)(4,2)
   \pscurve(2,2)(2,3)(2.5,3.5)(3.5,3.5)(4,2)
   \pscurve(0,2)(0,3)(1,3.8)(4,3.5)(4,2)
   \pscurve(0,2)(-1,2)(-2,3)(1,4.3)(5,3.5)(4,2)
   \pscurve(0,0)(-1,0)(-2.5,2.5)(-1.5,4)(4,4.5)(5.6,3.4)(4,2)
   \multiput(0,0)(2,0){3}{
   \pscircle[fillstyle=solid,fillcolor=white](0,0){.1}
   }
   \multiput(0,2)(2,0){3}{
   \pscircle[fillstyle=solid,fillcolor=white](0,0){.1}
   }
   \pscircle[fillstyle=solid,fillcolor=white](4,2){.4}
   \put(3.8,1.8){$f^*$}
   \put(3.,1.8){$l_1$}
   \put(3.8,1.){$l_2$}

   \put(0,-3.5){$H$ and $H^\perp$}

   }

   \put(9,0){
   \psline(0,0)(6,0)
   \psline(0,2)(6,2)
   \psline(0,4)(3,4)
   \psline(0,0)(0,4)
   \psline(2,0)(2,4)
   \psline(4,0)(4,3)
   \psline(6,0)(6,2)
   \multiput(0,0)(2,0){4}{
   \pscircle[fillstyle=solid,fillcolor=white](0,0){.1}
   }
   \multiput(0,2)(2,0){4}{
   \pscircle[fillstyle=solid,fillcolor=white](0,0){.1}
   }
   \multiput(0,4)(2,0){2}{
   \pscircle[fillstyle=solid,fillcolor=white](0,0){.1}
   }
   \pscircle[linestyle=dashed,fillstyle=solid,fillcolor=white](4,4){.3}
   \put(3.8,3.8){$v^*$}

   \put(1,1){
   \pscircle[linestyle=dashed,fillstyle=solid,fillcolor=white](4,2){.3}
   \put(3.8,1.8){$f^*$}
   \psline(-1,0)(5,0)
   \psline(-1,2)(3,2)
   \psline(0,-1)(0,3)
   \psline(2,-1)(2,3)
   \psline(4,-1)(4,1)
   \multiput(0,0)(2,0){3}{
   \pscircle[fillstyle=solid,fillcolor=white](0,0){.1}
   }
   \multiput(0,2)(2,0){2}{
   \pscircle[fillstyle=solid,fillcolor=white](0,0){.1}
   }
   }
   \multiput(0,0)(0,2){2}{
   \multiput(1,0)(2,0){3}{
   \pscircle[fillstyle=solid,fillcolor=black](0,0){.1}
   }
   }
   \multiput(0,1)(2,0){4}{
   \pscircle[fillstyle=solid,fillcolor=black](0,0){.1}
   }
   \multiput(0,3)(2,0){3}{
   \pscircle[fillstyle=solid,fillcolor=black](0,0){.1}
   }
   \multiput(1,4)(2,0){2}{
   \pscircle[fillstyle=solid,fillcolor=black](0,0){.1}
   }

   \put(2,-2.5){$N$}

   }

   \put(18,0){

   \psline(0,0)(4,4)
   \psline(0,2)(2,4)
   \psline(2,0)(5,3)
   \psline(4,0)(6,2)
   \psline(2,0)(0,2)
   \psline(4,0)(0,4)
   \psline(6,0)(2,4)
   \psline(6,2)(4,4)

   \psline(0,0)(6,0)
   \psline(0,2)(6,2)
   \psline(0,4)(4,4)
   \psline(0,0)(0,4)
   \psline(2,0)(2,4)
   \psline(4,0)(4,4)
   \psline(6,0)(6,2)
   \multiput(0,0)(2,0){4}{
   \pscircle[fillstyle=solid,fillcolor=white](0,0){.1}
   }
   \multiput(0,2)(2,0){4}{
   \pscircle[fillstyle=solid,fillcolor=white](0,0){.1}
   }
   \multiput(0,4)(2,0){3}{
   \pscircle[fillstyle=solid,fillcolor=white](0,0){.1}
   }

   \put(1,1){
   \psline(-1,0)(5,0)
   \psline(-1,2)(4,2)
   \psline(0,-1)(0,3)
   \psline(2,-1)(2,3)
   \psline(4,-1)(4,2)
   \multiput(0,0)(2,0){3}{
   \pscircle[fillstyle=solid,fillcolor=white](0,0){.1}
   }
   \multiput(0,2)(2,0){3}{
   \pscircle[fillstyle=solid,fillcolor=white](0,0){.1}
   }
   }
   \multiput(0,0)(0,2){2}{
   \multiput(1,0)(2,0){3}{
   \pscircle[fillstyle=solid,fillcolor=black](0,0){.1}
   }
   }
   \multiput(0,1)(2,0){4}{
   \pscircle[fillstyle=solid,fillcolor=black](0,0){.1}
   }
   \multiput(0,3)(2,0){3}{
   \pscircle[fillstyle=solid,fillcolor=black](0,0){.1}
   }
   \multiput(1,4)(2,0){2}{
   \pscircle[fillstyle=solid,fillcolor=black](0,0){.1}
   }
   \put(4.4,3.6){$e^*_1$}
   \put(5.4,2.6){$e^*_2$}

   \put(2,-2.5){$G$}

   }

   }
  \end{pspicture}
  \caption{By superimposing $H$ and
  its dual $H^\perp$ and removing 
  $f^*$ and $v^*$, we obtain a balanced bipartite
  graph $N$.
  $G$ is the normal graph induced by
  $V(N)\cup\{f^*,v^*\}$.}
  \label{fig:gmn}
 \end{center}
\end{figure}

Remove from $N'$ the vertex $f^*$ and a vertex $v^*\in V(H)$ 
which is adjacent to $f^*$ in $\Gamma$ and incident to the outer face of $H$.
Then we obtain a {\em balanced bipartite graph} $N$,
 which contains the same number of black and white vertices.
Burton and Pemantle \cite{BurtonPemantle} (see also \cite{KPW}) showed that
there is a bijection
between the set of dimer coverings
of $N$ and the set ${\mathcal T}$ of spanning trees of $H^\perp$.
Here, we review this bijection.
Let $T$ be a spanning tree of $H$.
Then the edges of $H^\perp$ that do not cross the edges of $T$
form a spanning tree of $H^\perp$, called the {\em dual tree}
and denoted by $T^\perp$.
This correspondence makes a bijection between
the set of spanning trees of $H$ and that of $H^\perp$.
We define the root of $T$ (resp. $T^\perp$)
to be $f^*$ (resp. $v^*$), and orient $T$ and $T^\perp$ so that
they point toward the roots.
Then the subset $M=\{\{x,\frac{x+y}{2}\}~|~ (x,y)\in T\mbox{ or }T^\perp \}$
of edges of $N$ is a dimer covering of $N$,
where $(x,y)$ denotes the oriented edge from $x$ to $y$.
This map $T\mapsto M$ is the bijection called
the Temperley bijection \cite{BurtonPemantle,KPW}.
Conversely, let $M$ be a dimer covering of $N$.
Then the map
\[
 \varphi:M\mapsto T=\left\{(x,y)~\bigg|~ \{x,\frac{x+y}{2}\}\in M, x\in V(H)\right\}.
\]
is the inverse of this bijection.

Let $G$ be the normal subgraph
of $\Gamma$ which is induced by the  vertices $V(G)=V(N)\cup \{v^*,f^*\}$.
(see Figure $\ref{fig:gmn}$.)
Then a dimer covering $M$ of $G$ contains exactly one impurity.
By Corollary $\ref{cor:slit-term}$, the slit-curve which intersects with the impurity
terminates at the middle points of the two
diagonal edges $e^*_1$ and $e^*_2$ on the boundary of $G$, 
each of which are adjacent to $f^*$.
Therefore, there exists a unique dimer covering $M'$
such that $M\stackrel{t}{\sim}M'$ and $e^*_1\in M'$.
Since $M'$ can be regarded as
the dimer covering of $N$, we obtain the following surjection,
\[
 \pi: {\mathcal M}(G) \ni  M \mapsto M' \in {\mathcal M}(N).
\]
\begin{lem}
 \label{lem:slit-temperley}
 Let $G$ be the graph described as above and
 let $M$ be a dimer covering of $G$. Then
 \[
  F(M)=\{\varphi\circ\pi(M)\}.
 \]
 Let $T^\perp$ be the dual tree of $\varphi\circ\pi(M)$.
 Then $F^\perp(M)$ can be obtained from $T^\perp$ by
 removing all edges incident to $f^*$ from $T^\perp$
 except for edges in $\{l_1,\ldots,l_{d^*}\}$.
\end{lem}
\begin{proof}
 After removing all edges intersecting with slit-curves from $N$,
 the resulting trees whose white vertices are in $V(\Lambda)$ have no
 black  leaves.
 Since $t$-moves keep the slit-curves unchanged,
 the slit-curves can not intersects with $T=\varphi(\pi(M))$ and $T^\perp$, 
 and $T$ is connected, hence $F(M)=\{\varphi\circ\pi(M)\}$.
\end{proof}

\begin{thm}
 Let $G$ and $H^\perp$ be graphs as described above.
 Let ${\mathcal M}$ be the set of the
 dimer coverings of $G$ and
 let ${\mathcal T}$ be the set of the
 spanning trees of $H$.
 Then the map 
 \[
  \bar{\varphi}:({\mathcal M}/\stackrel{t}{\sim})
 \ni [M]\mapsto\varphi\circ\pi(M)
 \in {\mathcal T}
\]
 is a bijection.
\end{thm}
\begin{proof}
 If $M_1\stackrel{t}{\sim}M_2$ then
 $\pi(M_1)=\pi(M_2)$.
 Thus $\bar{\varphi}$ is well-defined.
 Since $\pi$ is  surjective and $\varphi$ is a bijection,
 $\bar{\varphi}$ is also surjective.
 If $\bar{\varphi}([M_1])=\bar{\varphi}([M_2])$, then
 $\pi(M_1)=\pi(M_2)$ and hence $M_1\stackrel{t}{\sim}M_2$,
 that is, $\bar{\varphi}$ is injective.
\end{proof}

\subsection{Probability of finding the impurity at a given site}

\begin{figure}[H]
 \begin{center}
  \begin{pspicture}(6,4)

   \psline(0,0)(4,4)
   \psline(0,2)(2,4)
   \psline(2,0)(5,3)
   \psline(4,0)(6,2)
   \psline(2,0)(0,2)
   \psline(4,0)(0,4)
   \psline(6,0)(2,4)
   \psline(6,2)(4,4)

   \psline(0,0)(6,0)
   \psline(0,2)(6,2)
   \psline(0,4)(4,4)
   \psline(0,0)(0,4)
   \psline(2,0)(2,4)
   \psline(4,0)(4,4)
   \psline(6,0)(6,2)
   \multiput(0,0)(2,0){4}{
   \pscircle[fillstyle=solid,fillcolor=white](0,0){.1}
   }
   \multiput(0,2)(2,0){4}{
   \pscircle[fillstyle=solid,fillcolor=white](0,0){.1}
   }
   \multiput(0,4)(2,0){3}{
   \pscircle[fillstyle=solid,fillcolor=white](0,0){.1}
   }

   \psline[linecolor=black,linewidth=0.12,linestyle=solid]{*-*}(5,3)(3,3)(3,1)
   \psline[linecolor=black,linewidth=0.12,linestyle=solid]{*-*}(3,3)(1,3)

   \put(1,1){
   \psline(-1,0)(5,0)
   \psline(-1,2)(4,2)
   \psline(0,-1)(0,3)
   \psline(2,-1)(2,3)
   \psline(4,-1)(4,2)
   \multiput(0,0)(2,0){3}{
   \pscircle[fillstyle=solid,fillcolor=white](0,0){.1}
   }
   \multiput(0,2)(2,0){3}{
   \pscircle[fillstyle=solid,fillcolor=white](0,0){.1}
   }
   }
   \multiput(0,0)(0,2){2}{
   \multiput(1,0)(2,0){3}{
   \pscircle[fillstyle=solid,fillcolor=black](0,0){.1}
   }
   }
   \multiput(0,1)(2,0){4}{
   \pscircle[fillstyle=solid,fillcolor=black](0,0){.1}
   }
   \multiput(0,3)(2,0){3}{
   \pscircle[fillstyle=solid,fillcolor=black](0,0){.1}
   }
   \multiput(1,4)(2,0){2}{
   \pscircle[fillstyle=solid,fillcolor=black](0,0){.1}
   }

   \psarc[linewidth=0.075,linecolor=red](4,4){0.707}{-135}{-45}
   \psarc[linewidth=0.075,linecolor=red](3,3){0.707}{45}{135}
   \psarc[linewidth=0.075,linecolor=red](2,2){0.707}{-45}{45}
   \psarc[linewidth=0.075,linecolor=red](3,1){0.707}{135}{405}
   \psarc[linewidth=0.075,linecolor=red](2,2){0.707}{45}{135}
   \psarc[linewidth=0.075,linecolor=red](1,3){0.707}{45}{315}
   \psarc[linewidth=0.075,linecolor=red](4,2){0.707}{45}{225}
   \psarc[linewidth=0.075,linecolor=red](5,3){0.707}{225}{315}
   \psarc[linewidth=0.075,linecolor=red](2,4){0.707}{225}{315}

   \psline[linecolor=blue,linewidth=0.15]{*-*}(3,1)(2,0)

   \pscircle[fillstyle=solid,fillcolor=white](5,3){.3}
   \put(4.8,2.9){$f^*$}
   \put(4.4,3.6){$e^*_1$}
   \put(5.4,2.6){$e^*_2$}
  \end{pspicture}
  \caption{The impurity intersects with  the slit-curve $C^*$
  which connects the middle points of diagonal edges
  on the boundary.}
  \label{fig:slit-imp}
 \end{center}
\end{figure}
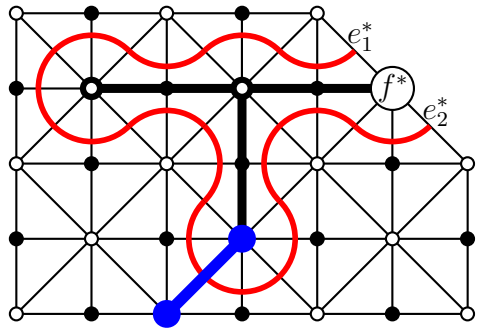

Let $e=\{x,y\}\in E_2$ be a diagonal edge in a normal graph $G$.
Then each class $[M]\in {\mathcal M}/\stackrel{t}{\sim}$ contains
at most one dimer covering with impurity $e$.
Thus, to count the dimer coverings with a fixed
impurity, we can instead count the trees of $H$ 
corresponding to such dimer coverings.
Let $M\in {\mathcal M}$ be a dimer covering such that $T=\bar{\varphi}([M])$.
Then $M$ determines the slit-curves, among which
the one $C^*$ intersecting with the impurity 
terminates at the middle points of $e^*_1=\{f^*, v^*\}$ and $e^*_2$.
Thus the slit-curve $C^*$ surrounds a tree $T^*\in F^\perp(M)$.
(See Figure $\ref{fig:slit-imp}$.)
Therefore, we have,
\begin{lem}
 A spanning tree $T$ of $H$ can be represented as
 $T=\varphi\circ\pi(M)$
 for some dimer covering $M$ containing impurity
 $e=\{x,y\}$ if and only if $x\in T^*$.
\end{lem}

 We choose a spanning tree $T$ of $H$ uniformly at random and
 define $p_v$ by 
 \begin{equation}
  \label{eq:pvdef}
  p_v={\rm Pr}(v \in T^*)
 \end{equation}
 for each $v\in V(H).$
 To obtain a uniformly random spanning tree of $H$,
 we can instead choose a uniformly random  spanning tree of $H^\perp$.
 By the last half of Lemma $\ref{lem:slit-temperley}$ and the results
 of Pemantle \cite{Pemantle} on 
 the random spanning trees and the loop erased random walks,
 $p_v$ is the probability
 of a simple random walk on $H^\perp$ starting at $v$
 to arrive at $f^*$ for the first time
 going through the edges in $\{l_1,\ldots,l_{d^*}\}$.
 By the definition $(\ref{eq:pvdef})$, it is clear $p_{f^*}=1$.
 The probabilities $p_v$'s can be computed via the {\em negative
 Laplacian} which is defined as follows.
 The negative Laplacian $A'=(a_{i,j})$
 of $H^\perp$ is the $\#V(H^\perp)$-dimensional square matrix
 defined as follows:
 Each of rows and coluumns of $A'$
 corresponds to a vertex of $H^\perp$ and 
\[
 a_{i,j}=
 \begin{cases}
  -1 & i \mbox{ and } j \mbox{ are adjacent,}\\
  0 &  i \mbox{ and } j \mbox{ are not adjacent,}\\
  4 & i=j, i\neq f^*,\\
  \mbox{the number of edges incident to }f^* & i=j=f^*.
 \end{cases}
\]
Let $A$ be the matrix obtained by
removing the row and column corresponding to $f^*$.
 Then the vector
 ${\mathbf p}=(p_w)_{w\in V\backslash\{f^*\}}$
 satisfies
 \[
  A{\mathbf p}={\mathbf b},
 \]
 where ${\mathbf b}=(b_w)_{w\in V\backslash\{f^*\}}$
 is defined by
 \[
  b_w=
 \begin{cases}
  1 & w \mbox{ is adjacent to }f^* \mbox{ in } \Lambda^\perp ,\\
  0 &  \mbox{ otherwise }.
 \end{cases}
 \]

\begin{thm}
 \label{thm:enum}
 Let $H$, $H^\perp$, $A$ and $p_v$ be as described above.
 Let $e=\{v,w\}$ be a diagonal edge of $G_{m,n}$,
 where $v\in H^\perp_{m,n}$.
 Then the number of dimer coverings
 containing the impurity $e$ is 
 \[
  |\det A|p_v.
 \]
 Therefore the total number of dimer coverings of $G$ is 
\[
 |{\mathcal M}|=|\det A|
 \left( 4\sum_{v\in V(H^\perp)} p_v + d^* - 3 \right)
 =
|\det A| \left(4\langle {\mathbf 1},A^{-1}{\mathbf b}\rangle+d^*+1\right),
\]
 where ${\mathbf 1}= (1,1,\ldots,1)^t$ and $d^*$ is the number of edges
 in $\Lambda^\perp$ connecting $f^*$ and other vertices in $V(H^\perp)$.
\end{thm}
\begin{proof}
 By  Kirchhoff's  Matrix-Tree Theorem (see Chapter 6 of \cite{Biggs}),
 the number of spanning trees of $H^\perp$ is $|\det A|$.
 Each $v\neq f^*\in V(H^\perp)$ is incident to
 exactly four diagonal edges of $G$ and
 $f^*$ is incident to $d^*-1$ diagonal edges.
\end{proof}

As an immediate corollary of Theorem $\ref{thm:enum}$ we have
the following:
\begin{thm}
 If we choose a dimer covering of $G$
 uniformly at random,
 the probability of finding a given
 impurity $e=\{v,w\}$ with $v\in V(H^\perp)$ is
 \[
  \frac{p_v}{4\sum_{w\in V(H^\perp)}p_w+d^*-3}.
 \]
\end{thm}

\begin{exm}
 Let $H$ be the graph consisting of $n$ squares in $\Lambda$ as shown in the left side of Figure
 $\ref{fig:onedim}$,
 and let $H^\perp$ have the vertex $f^*=(2n+1,1)$ and vertices
 indexed as shown in
 the middle of Figure $\ref{fig:onedim}$.
 Then we obtain the graph $G$ as shown in the right side of
 Figure $\ref{fig:onedim}$.
 Then we have
 \[
  A=
 \left(
 \begin{array}{cccccc}
  4 & -1 & & & & \\
  -1 & 4 & -1  & & & \\
  0 & -1 & 4 & -1  & & \\
   & &  & \ddots   & & \\
   &  &  &  & 4 & -1 \\
   &  &  &  &  & 4 \\
 \end{array}
 \right)
 \]
 The determinant of $A$ and the probability $p_j$ can be evaluated explicitly:
 \[
  \det A = \frac{1}{2\sqrt{3}}\left(\lambda_+^{n+1}-\lambda_-^{n+1}\right),
 \]
 where $\lambda_{\pm}=2\pm\sqrt{3}$.
 \[
  p_j=\frac{p_1}{2\sqrt{3}}\left(1-\left(\frac{\lambda_-}{\lambda_+}\right)^j\right),
 \]
 where
 $p_1=2\sqrt{3}\left(\lambda_+^{n+1}-\lambda_-^{n+1}\right)^{-1}$.
 Thus we have
 \[
  \sum_{j}p_j=\frac{p_1}{2\sqrt{3}}
 \left(
 \frac{\lambda_+^{n+1}-\lambda_+}{\lambda_+-1}
 -
 \frac{\lambda_-^{n+1}-\lambda_-}{\lambda_--1}
 \right)+2
 =\frac{p_1}{2\sqrt{3}}\lambda_+^n(1+o(1)).
 \]
 Thus the probability of finding an impurity $e=\{j,v\}$ in a random
 dimer covering is
 \[
  \frac{1}{4}\lambda_+^{-(n-j)}(1+o(1)).
 \]
\begin{figure}[H]
 \begin{center}
  \begin{pspicture}(16,2)(0,-1)
   \psset{unit=.5}
   \put(0,0){
   \psline(0,0)(4,0)
   \psline(6,0)(8,0)
   \psline[linestyle=dashed](4,0)(6,0)
   \psline(0,2)(4,2)
   \psline(6,2)(8,2)
   \psline[linestyle=dashed](4,2)(6,2)
   \multiput(0,0)(2,0){5}
   {
   \psline(0,0)(0,2)
   \pscircle[fillstyle=solid,fillcolor=white](0,0){.15}
   \pscircle[fillstyle=solid,fillcolor=white](0,2){.15}
   }
   \put(-1,-.8){$(0,0)$}
   \put(7,-.8){$(2n,0)$}
   \put(-1,2.3){$(0,2)$}
   \put(7,2.3){$(2n,2)$}
   \put(4,-2){$H$}
   }

   \put(12,0){
   
   {
   \psset{linecolor=lightgray,linestyle=dashed}
   \psline(0,0)(8,0)
   \psline(0,2)(8,2)
   \multiput(0,0)(2,0){5}
   {
   \psline(0,0)(0,2)
   }
   }
   \psline(6,1)(9,1)
   \psline(-1,1)(4,1)
   \put(0,0)
   {
   \psline(1,-1)(1,3)
   \pscircle[fillstyle=solid,fillcolor=white](1,1){.4}
   }
   \put(2,0)
   {
   \psline(1,-1)(1,3)
   \pscircle[fillstyle=solid,fillcolor=white](1,1){.4}
   }
   \put(6,0)
   {
   \psline(1,-1)(1,3)
   \pscircle[fillstyle=solid,fillcolor=white](1,1){.4}
   }
   \psline[linestyle=dashed](4,1)(6,1)
   \put(6.8,.75){$n$}
   \put(2.8,.75){$2$}
   \put(.8,.75){$1$}
   \pscircle[fillstyle=solid,fillcolor=white](9,1){.45}
   \put(8.6,.8){\small$f^*$}
   \put(7.6,.5){$l_1$}
   \put(4,-2){$H^\perp$}
   }

   \put(24,0){
   {
   \psline(0,0)(8,0)
   \psline(0,2)(8,2)
   \multiput(0,0)(2,0){5}
   {
   \psline(0,0)(0,2)
   }
   }
   \psline(0,1)(9,1)
   \psline(0,0)(1,1)(0,2)
   \multiput(0,0)(2,0){4}
   {
   \psline(1,0)(1,2)
   \psline(1,1)(2,0)(3,1)
   \psline(1,1)(2,2)(3,1)
   \pscircle[fillstyle=solid,fillcolor=white](1,1){.15}
   \pscircle[fillstyle=solid,fillcolor=white](0,0){.15}
   \pscircle[fillstyle=solid,fillcolor=white](2,0){.15}
   \pscircle[fillstyle=solid,fillcolor=white](0,2){.15}
   \pscircle[fillstyle=solid,fillcolor=white](2,2){.15}
   \pscircle[fillstyle=solid,fillcolor=black](0,1){.15}
   \pscircle[fillstyle=solid,fillcolor=black](2,1){.15}
   \pscircle[fillstyle=solid,fillcolor=black](1,0){.15}
   \pscircle[fillstyle=solid,fillcolor=black](1,2){.15}
   }
   \pscircle[fillstyle=solid,fillcolor=white](9,1){.15}
   \put(8.4,1.5){$e^*_1$}
   \put(8.4,0.1){$e^*_2$}
   \put(4,-2){$G$}
   }
  \end{pspicture}
  \caption{$H$ {\rm (left)}, $H^\perp$ \rm{(center)}, and $G$ {\rm (right)}}
  \label{fig:onedim}
 \end{center}
\end{figure}
\end{exm}

\begin{exm}
 Let $H$ be the graph as shown in the left side of Figure
 \ref{fig:m=n=1} and let $H^\perp$ be as shown in the
 middle of Figure \ref{fig:m=n=1}.
 Then we obtain $G$ shown in the right side of 
 Figure $\ref{fig:m=n=1}$ and we have
 \[
  A=\left(
 \begin{array}{ccc}
  4 & -1 & -1\\
  -1 & 4 & 0\\
  -1 & 0 & 4
 \end{array}
 \right),~~
 {\mathbf b}=
 (0,1,1)^t.
 \]
 Thus, $\det A=56$,
 ${\mathbf p}=A^{-1}{\mathbf b}=(\frac{1}{7},\frac{2}{7},\frac{2}{7})^t$ and 
 the number of
 dimer coverings is $328$.
 We have $56\times\frac{2}{7}=16$ dimer coverings with
 a fixed impurity incident to the vertex $3$,
 which are shown in Figure $\ref{fig:fixedimp}$.
\end{exm}

 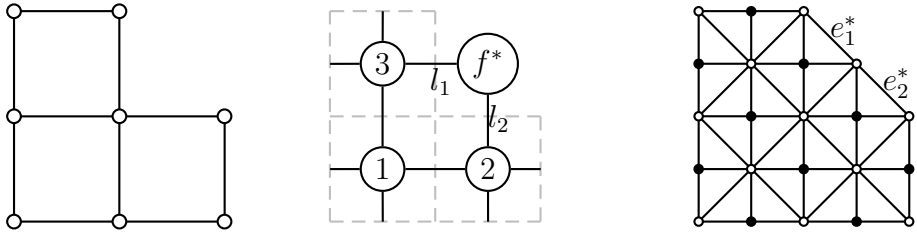
\begin{figure}[H]
  \begin{center}
  \begin{pspicture}(12,3)(0,-1)
   \psset{unit=0.7}
   \put(-1,-1){
   \psline(0,0)(4,0)
   \psline(0,2)(4,2)
   \psline(0,4)(2,4)
   \psline(0,0)(0,4)
   \psline(2,0)(2,4)
   \psline(4,0)(4,2)
   \put(0,0){\pscircle[fillcolor=white,fillstyle=solid](0,0){.15}}
   \put(2,0){\pscircle[fillcolor=white,fillstyle=solid](0,0){.15}}
   \put(4,0){\pscircle[fillcolor=white,fillstyle=solid](0,0){.15}}
   \put(0,2){\pscircle[fillcolor=white,fillstyle=solid](0,0){.15}}
   \put(2,2){\pscircle[fillcolor=white,fillstyle=solid](0,0){.15}}
   \put(4,2){\pscircle[fillcolor=white,fillstyle=solid](0,0){.15}}
   \put(0,4){\pscircle[fillcolor=white,fillstyle=solid](0,0){.15}}
   \put(2,4){\pscircle[fillcolor=white,fillstyle=solid](0,0){.15}}
   }
   \put(6,0){
   \put(-1,-1){
   \psline[linestyle=dashed, linecolor=lightgray](0,0)(4,0)
   \psline[linestyle=dashed, linecolor=lightgray](0,2)(4,2)
   \psline[linestyle=dashed, linecolor=lightgray](0,4)(2,4)
   \psline[linestyle=dashed, linecolor=lightgray](0,0)(0,4)
   \psline[linestyle=dashed, linecolor=lightgray](2,0)(2,4)
   \psline[linestyle=dashed, linecolor=lightgray](4,0)(4,2)
   }
   \cnodeput(0,0){A}{$1$}
   \cnodeput(2,0){B}{$2$}
   \cnodeput(0,2){C}{$3$}
   \cnodeput(2.,2.){D}{$f^*$}


   \pnode(-1.,0.){E}
   \pnode(0,-1.){F}
   \pnode(2.,-1){G}
   \pnode(3.,0.){H}
   \pnode(-1.,2.){I}
   \pnode(0.,3.){J}

   \ncline{A}{B}
   \ncline{A}{C}
   \ncline{B}{D}
   \ncline{C}{D}

   \ncline{A}{E}\ncline{A}{F}
   \ncline{B}{G}\ncline{B}{H}
   \ncline{C}{I}\ncline{C}{J}
   \put(.9,1.5){$l_1$}
   \put(2,.8){$l_2$}
   }

   \put(12,-1){

   \multiput(0,0)(0,1){3}{\psline(0,0)(4,0)}
   \psline(0,3)(3,3)
   \psline(0,4)(2,4)
   \multiput(0,0)(1,0){3}{\psline(0,0)(0,4)}
   \psline(3,0)(3,3)
   \psline(4,0)(4,2)

   \psline(0,0)(3,3)
   \psline(2,0)(4,2)
   \psline(0,2)(2,4)
   \psline(2,0)(0,2)
   \psline(4,0)(0,4)
   \psline(4,2)(2,4)

   \psset{fillstyle=solid,fillcolor=white}
   \multiput(0,0)(0,2){2}{
   \multiput(0,0)(2,0){3}{\pscircle(0,0){.1}}
   \multiput(1,1)(2,0){2}{\pscircle(0,0){.1}}
   }
   \multiput(0,4)(2,0){2}{\pscircle(0,0){.1}}

   \multiput(0,0)(0,2){1}{\psset{fillstyle=solid,fillcolor=black}
   \multiput(1,0)(2,0){2}{\pscircle(0,0){.1}}
   \multiput(0,1)(2,0){3}{\pscircle(0,0){.1}}
   }
   \multiput(1,2)(2,0){2}{\pscircle[fillstyle=solid,fillcolor=black](0,0){.1}}
   \multiput(0,3)(2,0){2}{\pscircle[fillstyle=solid,fillcolor=black](0,0){.1}}
   \multiput(1,4)(2,0){1}{\pscircle[fillstyle=solid,fillcolor=black](0,0){.1}}

   \put(2.5,3.5){$e^*_1$}
   \put(3.5,2.5){$e^*_2$}

   }
  \end{pspicture}
   \caption{$H$, $H^\perp$ and $G$}
   \label{fig:m=n=1}
  \end{center}
 \end{figure}

\begin{figure}[H]
 \begin{center}
  \includegraphics[width=15cm]{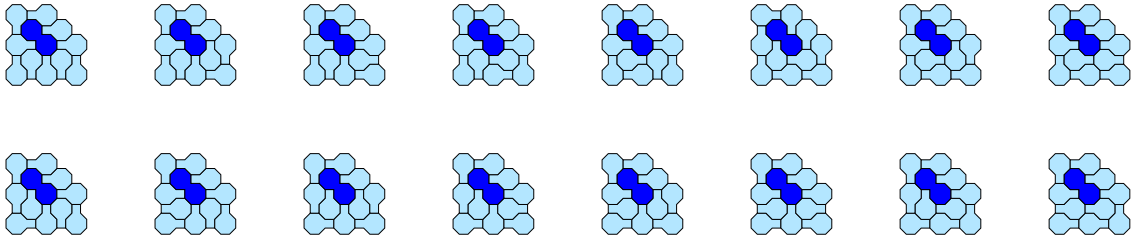}
 \end{center}
 \caption{Dimer coverings with a fixed impurity}
 \label{fig:fixedimp}
\end{figure}

\section*{Acknowledgment}
The authors thank Nicolas Destainville for helpful suggestions.


\end{document}